\documentclass[12pt,leqno,amsfonts,amscd]{amsart}
\usepackage{epsfig}
\usepackage{amsmath}
\usepackage{amsfonts}
\usepackage{amssymb}
\usepackage{color}
\usepackage{hyperref}

\newtheorem{theorem}{Theorem}[section]
\newtheorem{prop}[theorem]{Proposition}
\newtheorem{cor}[theorem]{Corollary}
\theoremstyle{definition}

\theoremstyle{remark}

\numberwithin{equation}{section}

\newcommand{\RR}{{\mathbb R}}
\newcommand{\CC}{{\mathbb C}}

\newcommand{\out}[1]{\ }

\let\PSH=\psh

\let\cal=\mathcal
\renewcommand{\phi}{\varphi}

\begin{document}

\title[A note on plurifinely open sets]{A note on the 
structure of plurifinely open sets and the 
equality of some complex Monge-Amp\`ere measures}

\author[M. El Kadiri]{Mohamed El Kadiri}
\address{University of Mohammed V
\\Department of Mathematics,
\\Faculty of Science
\\P.B. 1014, Rabat
\\Morocco}
\email{elkadiri30@hotmail.com}




\begin{abstract} In a recent preprint published on arXiv 
(see arXiv:2308.02993v2, referred here as \cite{NXH}),
N.X. Hong stated that every plurifinely open set $U\subset \CC^n$, $n\geq 1$, is of the form
$U=\bigcup \{\varphi_j>-1\}$, where each $\phi_j$ is a negative plurisubharmonic function
defined on an open ball $B_j\subset \CC^n$  and used this result
to prove an equality result on complex Monge-Amp\`ere measures. 
Unfortunately, this result is wrong as we will see below.
\end{abstract}
\maketitle

\section{Introduction} 
The plurifine topology on an open set $\Omega\subset 
\CC^n$, $n\geq 1$, is the coarsest topology on $\Omega$ 
that makes continuous all plurisubharmonic functions 
on $\Omega$. For properties of this topology we refer the reader to 
\cite{BT}. An open set relative to this topology is called a 
plurifinely open set. 
In a recent preprint on arXiv, cf. \cite{NXH}, N.X Hong gave a 
description of plurifinely open sets in $\CC^n$ ($n\geq 2$) 
and applied it to establish a result on complex Monge-Amp\`ere 
measures. More precisely, he stated the following results:

\begin{theorem}[{\cite[Theorem 1.1]{NXH}}]\label{thm1.1}
Let $\Omega$ be a subset of $\CC^n$ and let $\{z_j\}\subset \Omega$ 
be such that $\overline {\{z_j\}} =\Omega$. Then, $\Omega$ is
plurifinely open if and only if there exists a sequence of negative 
plurisubharmonic functions $\varphi_{j,k})$ in $B(z_j, 2^{-k})$ such that
$$\Omega=\bigcup_{j,k=1}\{\varphi_{j,k} >-1\}.$$
\end{theorem} 

\begin{cor}
Every plurifinely open set is Borel.
\end{cor}

\begin{prop}[{\cite[Proposition 2.2]{NXH}}]\label{prop1.3}
Let $D$ be a Euclidean open set and let $\PSH(D)$ be the 
set of all plurisubharmonic functions in $D$. Then, for 
every plurifinely open set $\Omega$, the function
$$u := \sup\{\varphi\in \PSH(D) : \varphi\leq  -1 \text{ on } D\setminus \Omega 
\text{ and } \varphi \leq 0  \text{ in } D\}$$
is plurisubharmonic in $D$ and
$u = -1$  on $D\setminus \Omega$.
\end{prop}

\begin{theorem}\label{thm1.4}
Let $D$ be a bounded hyperconvex domain and let 
$\Omega$ be a plurifinely open subset of $D$. Assume that 
$u_0, \cdots . . . , u_n \in \cal E(D)$ such that $u_0 = u_1$ on $\Omega$. 
Then
\begin{equation}\label{eq1.2} 
dd^cu_0 \wedge T = ddcu_1 \wedge T
\end{equation} 
on $\Omega$. Here $T := dd^cu_2 \wedge . . . \wedge dd^cu_n.$
\end{theorem}

Note that Theorem 1.1 and Theorem 1.2
in \cite{NXH} are respectively Theorem 1.1 and Theorem 1.2 of the first version
 arXiv:2308.02993v1 of \cite{NXH}.

\section{Theorem 1.1 is wrong}

Theorem \ref{thm1.1} and its corollary are wrong 
as we will show by giving an example of a plurifinely open subset of 
$\CC^n$ that is not a Borel set.

Recall that for $n=1$
pluri-polar  and plurifinely open subsets of $\CC^1=\CC$ are
just the classical polar subsets and finely open
subsets of $\CC\cong \RR^2$. Any polar set $E\subset \CC$ is finely closed 
(in $\CC^n$, $n\geq 2$, this property is not true, a pluripolar set 
is not necessarily plurifinely closed) and
there are polar subsets of $\CC$ that are not countable,
see \cite[I.V.3, p. 59]{D}. Let $E\subset \CC$ be
a non countable polar set and denote by $\cal B(\CC)$ the $\sigma$-algebra of 
Borel subsets of $\CC$. It is well known that 
$\# \cal B(\CC)=\mathfrak{c} <2^{\mathfrak c}=\#\mathfrak P(E)$ (the latter equality
holds because $\#E=\mathfrak c$). It then follows that
$\mathfrak P(E)\nsubseteq \cal B(\CC)$ and, consequently, 
there are subsets of $E$ that are not Borel subsets of $\CC$. Since every subset of $E$ is a polar, 
we conclude that there are  polar subsets of $\CC$ that are not Borel sets. 
In the same we can show that there are pluripolar subsets of $\CC^n$, $n\geq 2$, that are 
not Borel sets.

Let $F$ be a polar subset of $\CC$ that is not a Borel set 
and let $U=\CC\setminus F$. Then $U$ is a
finely open subset of $\CC$ which is not a Borel set.
Let $n\geq 2$, then $U^n$ is a plurifinely open subset of $\CC^n$. 
The function
$f:\CC \rightarrow \CC^n$ defined by $f(z)=(z,\cdots, z)$ is  
Borel measurable (because it is continuous) and hence 
$U^n\subset \CC^n$ is not a Borel set because otherwise 
$U=f^{-1}(U^n)$ would be a Borel subset of $\CC$. 
In the same way $U\times \CC^{n-1}$ is not a Borel subset of $\CC^n$.
We then have proved that Theorem \ref{thm1.1} is wrong for every $n\geq 1$.

\section{Proposition 1.3 is wrong}

Proposition 1.3, on which is based  the proof of the (wrong) theorem \ref{thm1.1},
is also wrong. Here we give a concrete example
showing this. Let $U \subset  \CC$ be
an irregular bounded domain such that
$U\subset B(0,1)$, and let $z\in \partial U$ be an irregular point.
For every (Euclidean) domain $O\subset B(0,1)$ and every $A\subset O$ we denote by
${^OR}_f^A$ and ${^O\widehat R}_f^A$ 
(or ${R}_f^A$ and ${\widehat R}_f^A$ if there is no risk of confusion) 
respectively, the reduced and the balayage of
a function $f:O\rightarrow [0,+\infty]$ relative to $O$ (see \cite[p. 129]{AG}), 
and by $\cal S_+(O)$ the cone of nonnegative superharmonic 
functions on $O$. Recall that
$$^OR_f^A=\inf\{u\in \cal S_+(O): u\geq f \text{ on } A\}$$
and that ${^O\widehat R}_f^A$ is the l.s.c. regularized function of
$^OR_f^A$, that is, the function defined by
$${^O\widehat R}_f^A(z)=\liminf_{\zeta\to z} {^OR}_f^A(\zeta)$$
for every $z\in O$. It is clear $O$ and $O'$ are two 
open sets such that $O\subset O'\subset B(0,1)$
we have $^OR_f^A\leq {^{O'}R}_f^A|_O$ and $^O{\widehat R}_f^A\leq {^{O'}}{\widehat R}_f^A|_O$.

Since $z$ is an irregular point of $U$, it follows by
\cite[Theorem 7.3.1 (ii)]{AG} that there is an open ball 
$B(z,r)\subset \overline B(z,r)\subset B(0,1)$
such that
$$\widehat R_1^{B(z,r)\setminus U}(z):= {^{B(0,1)}\widehat R}_1^{B(z,r)\setminus U}(z)<1$$
and hence
\begin{equation}\label{eq1}
{^{B(z,r)}}\widehat R_1^{B(z,r)\setminus U\cap B(0,r)}(z)=
{^{B(z,r)}}\widehat R_1^{B(z,r)\setminus U}(z)<1.
\end{equation}

Take $D=B(z,r)$ and $\Omega=U\cap B(z,r)$ and let $u$ be  the function
$$u:=\sup\{\varphi\in \PSH_-(D): u\leq -1 \text{ on } D\setminus U\}.$$ 
We have 
$${^{B(0,r)}}\widehat R_1^{B(z,r)\setminus \Omega}=-u^*$$ 
and by (\ref{eq1}) that $u(z)<u^*(z).$ This proves that Proposition 2.2 is wrong. 

\section{The correct form of Theorem 1.4} 
As it is stated, Theorem \ref{thm1.4} is not sure, because the given proof of this theorem
in \cite{NXH} is based on the wrong theorem \ref{thm1.1}. However, that proof 
is correct if $\Omega$ has the form given in Theorem \ref{thm1.1}. So, Theorem 
\ref{thm1.4} should be stated correctly as follows: 

\begin{theorem}\label{thm1.5}
Let $D\subset \CC^n$ ($n\geq 1$) be a bounded hyperconvex domain and let
$\Omega$ be a plurifinely open subset of $D$ of the form 
$\Omega=\bigcup_j\{\varphi_j>0\}$, where each $\phi_j$ is 
a negative plurisubharmonic function
defined on an open ball $B_j\subset D$. Assume that
$u_0,. . . , u_n \in \cal E(D)$ such that $u_0 = u_1$ on $\Omega$.
Then
\begin{equation}\label{eq1.5}
dd^cu_0 \wedge T = dd^cu_1 \wedge T
\end{equation}
on $\Omega$. Here $T := dd^cu_2 \wedge . . . \wedge dd^cu_n.$
\end{theorem}

\begin{proof}
It suffices to prove the theorem in case $\Omega=\{\varphi>0\}$, 
where $\phi$ is a nonpositive plurisubharmonic function
on an open ball $B_j\subset D$. To do this, we bring from \cite{NXH} 
the part of the proof of Theorem 
\ref{thm1.4} corresponding to this case (which is correct). Without loss
of generality, we may assume that $u_k \in {\cal F}(D)$, $0 \leq k \leq n$. 
Let $\psi \in {\cal E}_0(D)$ and define
$u_{j,k}:= \max\{j\psi, u_k\}$, $k = 0, 1, ..., n.$ 
We easily have 
$${\cal E}_0(\Omega) \ni u_{j,k}\downarrow u_k \text{ as } j\uparrow +\infty, 
\  k = 0, ...,, n$$
and, by the hypotheses,
\begin{equation}\label{eq1.6}
u_{j,0} = u_{j,1} \text{ on } \Omega,  \forall j \geq 1.
\end{equation}
Write $T_j:= dd^cu_{j,2} \wedge ... \wedge dd^cu_{j,n}.$ Since 
the $u_{j,k}$ are bounded plurisubharmonic functions in $D$, 
it follows by Corollary 3.4 in \cite{EKW} that
$$dd^cu_{j,0} \wedge T_j = dd^cu_{j,1} \wedge T_j \text{ on } \Omega.$$ 
and hence 
\begin{equation}\label{eq1.7}
\varphi dd^cu_{j,0} \wedge T_j = \varphi dd^cu_{j,1} \wedge T_j \text{ on } D
\end{equation} 
because $\varphi=0$ on $D\setminus \Omega$. By Corollary 5.2 in \cite{Ce} we have 
$$\varphi dd^c\wedge u_{j,0} \wedge T_j \to \varphi dd^cu_0 \wedge T$$ 
and 
$$\varphi dd^c\wedge u_{j,1} \wedge T_j \to \varphi dd^cu_1 \wedge T$$ 
weakly* on $D$ as $j \to +\infty$. Hence, by (\ref{eq1.7}), 
$$\varphi dd^cu_0 \wedge T = \varphi dd^cu_1 \wedge T \text{ on } D,$$
and therefore
$$dd^cu_0 \wedge T = dd^cu_1 \wedge T \text{ on } \Omega.$$
\end{proof}

Theorem \ref{thm1.5} is a partial extension of \cite[Theorem 4.5]{EK}. 
In a forthcoming paper we shall extend Theorem \ref{thm1.5} to every 
plurifnely open subset of a given bounded hyperconvex domain $D\subset \CC^n$.

\end{document}